\documentclass[12pt]{amsart}
\usepackage{amssymb,mathrsfs,color,bbold}

\makeatletter

\newcommand{\Rmnum}[1]{\expandafter\@slowromancap\romannumeral #1@}
\makeatother
\usepackage{mathtools}

\theoremstyle{plain}
\newtheorem{theorem}{Theorem}[section]
\newtheorem{lemma}[theorem]{Lemma}
\newtheorem{proposition}[theorem]{Proposition}
\newtheorem{corollary}[theorem]{Corollary}

\theoremstyle{definition}
\newtheorem{definition}[theorem]{Definition}

\newtheorem{example}[theorem]{Example}

\title[st-unbounded-convergence]{statistical unbounded convergence in Banach lattices}

\date{\today}


\author[Z. Wang]{Zhangjun Wang}
\address{School of Mathematics, Southwest Jiaotong University,
Chengdu, Sichuan,
China, 610000.}
\email{zhangjunwang@my.swjtu.edu.cn}

\author[Z. Chen]{Zili Chen}
\address{School of Mathematics, Southwest Jiaotong University, Chengdu, Sichuan,
China, 610000.}
\email{zlchen@swjtu.edu.cn}

\author[J. Chen]{Jinxi Chen}
\address{School of Mathematics, Southwest Jiaotong University, Chengdu, Sichuan,
	China, 610000.}
\email{jinxichen@swjtu.edu.cn}

\begin{document}

\begin{abstract}
Statistical and unbounded convergence have been investigated in the literature. In this paper, we introduce the statistical unbounded order and topology convergence in Banach lattices. We first present some characterizations and related properties for these convergences. Also we using the convergences to characterize order continuous, $KB$ and reflexive Banach lattices.
\end{abstract}

\maketitle

\section{introduction}
Several recent papers investigated unbounded convergence. A net $(x_\alpha)$ in Riesz space $E$ is unbounded order convergent ($uo$-convergent for short) to $x$ if $|x_\alpha-x|\wedge u\xrightarrow{o}0$ for all $u\in E_+$. Similarly, there are unbounded norm convergence, unbounded absolute weak convergence and unbounded absolute weak* convergence in Banach lattices \cite{DOT:16,O:16,W:19}. For further properties about these convergences, see [5-10,13] for details.

Statistical convergence has been studied in functional analysis literature, and also been researched in Riesz spaces and Banach lattices \cite{SCS:12,XT:18}. In the past, statistical convergence was usually defined by sequences. In \cite{MDG:19}, this convergence could be studied for special nets.

Combine these convergences, we introduce so called  statistical unbounded convergence. We first establish the basic properties of statistical unbounded order convergence, statistical unbounded norm convergence, statistical unbounded absolute weak convergence and statistical unbounded absolute weak* convergence in Banach lattices, then we using these convergences to describe order continuous Banach lattices, $KB$-spaces and reflexive Banach lattices. 

For undefined terminology, notations and basic theory of Riesz spaces and Banach lattices, we refer to [1-4].

\section{statistical unbounded order convergence}\label{}
We start with the definition of statistical order convergence. Let $K\subset N$, the \emph{natural density} of $K$ is defined by
$$\delta(K)=\lim_{n\rightarrow\infty}\frac{1}{n}
|\{k\in K: k\leq n\}|$$
where $|\cdot|$ denotes the cardinality of the set. 
A sequence $(x_n)$ in Riesz space $E$ is statistical monotone increasing (resp. decreasing) if $(x_k)$ is increasing (resp. decreasing) for $k\in K$ with $\delta(K)=1$, we write $x_n\uparrow^{st}$ (resp. $x_n\downarrow^{st}$). A increasing or decreasing sequence will be called a \emph{statistical	monotonic sequence}, moreover, if $\sup_{k\in K}x_k=x$ (resp. $inf_{k\in K}x_k=x$) for some $x\in E$, then $(x_n)$ is \emph{statistical monotone convergent} to $x$, we write $x_n\uparrow^{st}x$ (resp. $x_n\downarrow^{st}x$). A sequence $(x_n)$ in $E$ is \emph{statistical order convergent} to $x\in E$ provided that there exists a sequence $p_n\downarrow^{st}0$ such that $|x_n-x|\leq p_n$, we call $x$ the statistical order limit of $(x_n)$ ($x_n\xrightarrow{st-ord}x$ for short), it is easy to get that $\delta\{n\in N:|x_n-x|\nleq p_n\}=0$. 

For a upwards directed set $(D,\leq)$, if for each $\beta\in D$, the set $\{\alpha\in D:\alpha\leq\beta\}$ is finite, we call it \emph{lower finite set}. To a nonempty subset $A$ of $D$, the \emph{asymptotic density} denoted by $\delta(A;D)=\lim_{\alpha\in D}\frac{|A\cap D_{\alpha}|}{|D_{\alpha}|}$, where $|D_{\alpha}|$ denotes the cardinality of $\{\beta\in D :\beta\leq\alpha\}$. If the net is defined on a lower finite set, we call it \emph{lower finite net}. We are easy to find that sequence is a special lower finite net. If $(x_n)$ (resp. $(x_\alpha)$) is a sequence (resp. lower finite net) such that $(x_n)$ (resp. $(x_\alpha)$) satisfies property $P$ for all $n$ (resp. $\alpha$) except a set of natural density zero, then we say that $(x_n)$ (resp. $(x_\alpha)$) satisfies property $P$ for almost all
$n$ (resp. $\alpha$) and we abbreviate this by $a.a.n$ (resp. $a.a.\alpha$). According to these definition, we can define statistical convergence for lower finite net. In this paper, we use $D$ denote lower finite set.

A net $(x_\alpha)$ in Riesz space $E$ is \emph{unbounded order convergent} to $x\in E$ if $(|x_\alpha-x|\wedge u)$ is order convergent for any $u\in E^+$. Combine these convergences, we introduce a new convergence in Riesz space.
\begin{definition}\label{}
A sequence $(x_n)$ in a Riesz space $E$ is \emph{statistical unbounded order convergent} to $x\in E$ if  $(|x_n-x|\wedge u)$ is statistical order convergent for all $u\in E^+$ is statistical order convergent (denoted by $x_n\xrightarrow{st-uo}x$). Respectively, for a lower finite net $(x_\alpha)$, if $(|x_\alpha-x|\wedge u)$ is statistical order convergent for all $u\in E^+$, we call $x$ the statistical unbounded order limit of $(x_\alpha)$ ($x_\alpha\xrightarrow{st-uo}x$ for short).	
\end{definition}

\begin{proposition}\label{}
For lower finite net $(x_\alpha)$ and $(y_\alpha)$ in a Riesz space $E$, we have the following results:

$(1)$ the $st$-$uo$-limit of $(x_\alpha)$ is unique;

$(2)$ suppose $x_\alpha\xrightarrow{st-uo}x$ and $y_\alpha\xrightarrow{st-uo}y$, then $ax_\alpha+by_\alpha\xrightarrow{st-uo}ax+by$ for any $a,b\in R$;

$(3)$ $x_\alpha\xrightarrow{st-uo}x$ iff $(x_\alpha-x)\xrightarrow{st-uo}0$;

$(4)$ if $x_\alpha\xrightarrow{st-uo}x$,then $|x_\alpha|\xrightarrow{st-uo}|x|$.
\end{proposition}
\begin{proof}
(1) Let $(x_\alpha)$ be a sequence in $E$ such that $x_\alpha\xrightarrow{st-uo}x$ and $x_\alpha\xrightarrow{st-uo}y$. Then
we can find $(p_\alpha)$, $(q_\alpha)$ such that $p_\alpha\downarrow^{st}0$ and $q_\alpha\downarrow^{st}0$, and a set $B\subset D$ with $\delta(B;D)=1$ such that for $\beta$ of the index subset $B$ of $D$, we have:
$$|x_{\beta}-x|\wedge u\leq p_{\beta}\downarrow0, |x_{\beta}-y|\wedge u\leq q_{\beta}\downarrow0.$$
Thus, we get
$$0\leq|x-y|\wedge u\leq|x_{\beta}-x|\wedge u+|x_{\beta}-y|\wedge u\leq p_{\beta}+q_{\beta}$$
for every $\beta\in B$ and $u\in E_+$, which shows that $x=y$;

(2) and (3) is obviously; 

(4) According to $\big||x_\alpha|-|x|\big|\wedge u\leq|x_\alpha-x|\wedge u$, we have the conclusion.
\end{proof}

It is similar to Theorem 7 of \cite{SCS:12} that Squeeze law holds for statistical unbounded order convergence.
\begin{proposition}\label{}
Let $E$ be a Riesz space; $(x_\alpha)$, $(y_\alpha)$ and $(z_\alpha)$ be lower finite nets in $E$ such that $x_\alpha\leq y_\alpha\leq z_\alpha$ for all $\alpha\in D$, if $x_\alpha\xrightarrow{st-uo}x$ and $z_\alpha\xrightarrow{st-uo}x$, then $y_\alpha\xrightarrow{st-uo}x$.	
\end{proposition}
The $st$-$uo$-convergence preverse disjoint property.
\begin{proposition}\label{}
Let $x_\alpha\xrightarrow{st-uo}x$ and $\delta({\alpha\in D:|x_\alpha|\wedge |y|\neq0})=0$, then $x\perp y$.	
\end{proposition}
\begin{proof}
Since $x_\alpha\xrightarrow{st-uo}x$, we have $|x_\alpha|\xrightarrow{st-uo}|x|$. According to $|x_\alpha|\wedge |y|=0$ for $a.a.\alpha$, and $|x_\alpha|\wedge |y|\xrightarrow{st-uo}|x|\wedge|y|$, it follows that $|x|\wedge|y|=0$, $i.e.$$x\perp y$.	
\end{proof}

In general, any subsequence of a convergent sequence is convergent, but it does not hold for $st$-$uo$-convergence.
\begin{example}\label{}
In $l_p (1\leq p<+\infty)$, let $(e_n)$ be the 	standard basis, and 
\[x_n=\begin{dcases}
\sum_{1}^{k}e_{i^2},& n=k^2(k\in N) ;\\
e_n,& ortherwise.
\end{dcases}  \]
Since $(e_n)$ is disjoint sequence, by Corollary 3.6 of \cite{GTX:16}, so it is unbounded order convergent to $0$, moreover, $(x_n)$ is statistical unbounded order convergent to $0$, but the subsequence $(x_{k^2})$ does not. We also can find that $(x_n)$ is not $st$-order convergent and $uo$-convergent.
\end{example}

In Example 2.5, we know that $st$-$uo$-convergence is not necessarily $uo$-convergence. When does the $st$-$uo$-convergence imply unbounded order convergence? Next result show that the relation between two convergences.
\begin{proposition}\label{}
Let $(x_\alpha)$ be a lower finite net in Riesz space $E$, then $(x_\alpha)$ is statistical unbounded order convergent to $x\in E$ if and only if there is another lower finite net $(y_\alpha)$ such that $x_\alpha=y_\alpha$ for $a.a.\alpha$ and which is unbounded order convergent to the same limit $x$.	
\end{proposition}

\begin{proof}
$\Rightarrow$ Following from definition.

$\Leftarrow$ Assume that $x_\alpha=y_\alpha$ for $a.a.\alpha$ and $y_\alpha\xrightarrow{uo}x$. Then there exists a lower finite $(p_\alpha)\downarrow0$ (hence $p_\alpha\downarrow^{st}0$) and $|y_\alpha-x|\wedge u\leq p_\alpha$ for every $u\in E_+$. Thus, we can write
$$\{\alpha\leq\beta:|x_\beta-x|\wedge u\nleq p_\beta\}\subset$$
$$\{\alpha\leq\beta:x_\beta\ne y_\beta\}\cup\{\alpha\leq\beta:|y_\beta-x|\wedge u\nleq p_\beta\}$$	
Since $y_\alpha\xrightarrow{uo}x$, the second set on the right side is empty. Hence we have $\delta\{\alpha\in D:|x_\alpha-x|\wedge u\nleq p_\alpha\}=0$, $i.e$: $x_\alpha\xrightarrow{st-uo}x$.
\end{proof}

For monotonic lower finite net, we have the further result.

\begin{proposition}\label{}
For monotonic lower finite net $(x_\alpha)$ in Riesz space $E$, $x_\alpha\xrightarrow{uo}x$ iff $x_\alpha\xrightarrow{st-uo}x$.
\end{proposition}
\begin{proof}
Assume that $x_\alpha\uparrow$ and $x_\alpha\xrightarrow{st-uo}x$, then there exists a $B\subset D$, all $\beta\in B$ and $\delta(B;D)=1$ such that $(x_\beta)$ is a subnet of $(x_\alpha)$ and $\sup_{\beta\in B}x_{\beta}=x$. For all $\gamma\in B^c$, since $x_{\gamma}\uparrow$, so we have $x_{\gamma}\uparrow\leq x$, factly, if $x_{\gamma_0}>x$, then there exists a $\eta_0$ such that $x_{\eta_0}\geq x_{\gamma_0}>x$, it is contradiction. And since for any $\gamma$ we can find a $\beta$ such that $x_{\beta}\leq x_{\gamma}\leq x$, by $\sup_{\beta\in B}x_{\beta}=x$, it follow that $\sup_{\gamma\in B^c}x_{\gamma}=x$, hence $x_{\gamma}\xrightarrow{uo}x$ for $\gamma\in B^c$, so $x_\alpha\xrightarrow{uo}x$.
\end{proof}

Let $(x_n)$ be a lower finite net in Riesz space $E$, $(x_\alpha)$ is called \emph{statistical order bounded} if there exists an order interval $[x,y]$ such that
$\delta\{\alpha\in D:x_\alpha\notin [x,y]\}=0$.
It is clear that every order bounded lower finite net is statistical order bounded, but the converse is not true in general.
\begin{example}\label{}
In $l_\infty$, let $(e_n)$ be the unit vectors, and 
\[x_n=\begin{dcases}
ne_n,& n=k^2(k\in N) ;\\
e_n,& ortherwise.
\end{dcases}  \]
$(x_n)$ is statistical order bounded sequence, but it is not order bounded.
\end{example}

According to Example 2.5, $st$-$uo$-convergence is not necessarily statistical order convergence. It is natural to ask when the $st$-$uo$-convergence implies $st$-$ord$-convergence? Next result shows that the $st$-$uo$-convergence is the same as $st$-$ord$-convergence for statistical order bounded lower finite net.

\begin{proposition}\label{}
Let $(x_\alpha)$ be a statistical order bounded lower finite net in Riesz space $E$ and $x_\alpha\xrightarrow{st-uo}x$, then $x_\alpha\xrightarrow{st-ord}x$.	
\end{proposition}
\begin{proof}
Since $x_\alpha\xrightarrow{st-uo}x$, therefore, $x_\alpha\xrightarrow{uo}x$ for $a.a.\alpha$. And since $(x_\alpha)$ is statistical order bounded lower finite net, then there exists $u\in E_+$ such that $|x_\alpha|\leq u$ for $a.a.\alpha$, then fix the $u$, we have $|x_\alpha-x|\wedge 2u=|x_\alpha-x|\xrightarrow{o}x$ for $a.a.\alpha$, therefore, $x_\alpha\xrightarrow{st-ord}x$.	
\end{proof}

A norm on a Banach lattice $E$ is called \emph{order continuous} if $\Vert x_\alpha\Vert\rightarrow0$ for $x_\alpha\downarrow0$. A Banach lattice $E$ is called a \emph{$KB$-space}, if every monotone bounded sequence is convergent. A lower finite net $(x_\alpha)$ is said to be statistical unbounded order Cauchy, if $(x_\alpha-x_\beta)$ $st$-$uo$-convergent to 0.
Let $E$ be a Banach space and $(x_\alpha)$ be a lower finite net in $E$, $(x_\alpha)$ is \emph{statistical norm bounded} if for any $C\in R^+$ there exists a $\lambda>0$ such that
$\delta\{\alpha\in D:\Vert \lambda x_\alpha\Vert>C\}=0$. In Example 2.8, it is clear that statistical norm bounded is not necessarily bounded.

According to Theorem 4.7 of \cite{GX:14}, we decribe $KB$ spaces by $st$-$uo$-convergence.
\begin{theorem}\label{}
Let $E$ be an order continuous Banach lattice. The following are equivalent:

$(1)$ $E$ is $KB$ space;

$(2)$ every norm bounded $uo$-Cauchy sequence in $E$ is $uo$-convergent;

$(3)$ every statistical norm bounded $st$-$uo$-Cauchy sequence in $E$ is $st$-$uo$-convergent.
\end{theorem}
\begin{proof}
$(1)\Rightarrow(2)$ by Theorem 4.7 of \cite{GX:14} and $(2)\Rightarrow(3)$ is clearly;

$(3)\Rightarrow(1)$ Assume that $E$ is not $KB$ space, then $E$ contains a sublattice lattice isomorphic to $c_0$. Let $(x_n)$ be a sequence in $c_0$, and
\[x_n=\begin{dcases}
ne_n  ,& n=k^2(k\in N) ;\\
\sum_{1}^{n}e_i, & ortherwise.
\end{dcases}  \]
$(x_n)$ is statistical norm bounded and $st$-$uo$-Cauchy in $E$ but it is not $st$-$uo$-convergent.
\end{proof}
A sequence $(x_n)$ in Banach space $E$ is \emph{statistical norm convergent} to $x\in E$ provided that $\delta\{n\in N:\Vert x_n-x\Vert\geq \epsilon\}=0$ for all $\epsilon>0$. We write $x_n\xrightarrow{st-n}x$. We call $x$ the statistical norm limit of $(x_n)$.
A sequence $(x_n)$ in $E$ is \emph{statistical weak convergent} to $x\in E$ provided that $\delta\{n\in N:|f(x_n-x)|\geq \epsilon\}=0$ for all $\epsilon>0$, $f\in E^{'}$. we write $x_n\xrightarrow{st-w}x$. We call $x$ the statistical weak limit of $(x_n)$.
A sequence $(x_n^{'})$ in $E^{'}$ is \emph{statistical weak* convergent} to $x^{'}\in E^{'}$ provided that $\delta\{n\in N:|(x_n^{'}-x^{'})(x)|\geq \epsilon\}=0$ for all $\epsilon>0$, $x\in E$. we write $x_n^{'}\xrightarrow{st-w^{*}}x$. We call $x^{'}$ the statistical weak limit of $(x_n^{'})$.

Weak convergent sequence is $st$-$w$-convergent and statistical norm bounded, the converse is not hold in general by Example 2.5 for ($1<p<+\infty$). Now, we consider a Banach space with two properties. Let $E$ be a Banach space, $E$ is said to have \emph{statistical weak convergent property} ($STWC$ property, for short) if $(x_n)$ is statistical norm bounded and $x_n\xrightarrow{st-w}0$ in $E$, then $x_n\xrightarrow{w}0$ for all $(x_n)\subset E$. $E$ is said to have \emph{statistical weak* convergent property} ($STW^{*}C$ property, for short) if $(x_n^{'})$ is statistical norm bounded and $x_n^{'}\xrightarrow{st-w^{*}}0$ in $E^{'}$, then $x_n^{'}\xrightarrow{w^{*}}0$. Using these properties, we investigate order continuous Banach lattices by $st$-$uo$-convergence.
\begin{theorem}\label{}
Let $E$ be an Banach lattice:

$(1)$ $E$ has order continuous norm;

$(2)$ for any statistical norm bounded lower finite net $(x^{'}_\alpha)$ in $E^{'}$,if $x^{'}_\alpha\xrightarrow{st-uo}0$,then $x^{'}_\alpha\xrightarrow{st-\sigma(E^{'},E)}0$;

$(3)$ for any statistical norm bounded lower finite net $(x^{'}_\alpha)$ in $E^{'}$,if $x^{'}_\alpha\xrightarrow{st-uo}0$,then $x^{'}_\alpha\xrightarrow{st-|\sigma|(E^{'},E)}0$;

$(4)$ for any statistical norm bounded sequence $(x^{'}_n)$ in $E^{'}$,if $x^{'}_n\xrightarrow{st-uo}0$,then $x^{'}_n\xrightarrow{st-\sigma(E^{'},E)}0$;

$(5)$ for any statistical norm bounded sequence $(x^{'}_n)$ in $E^{'}$,if $x^{'}_n\xrightarrow{st-uo}0$,then $x^{'}_n\xrightarrow{st-|\sigma|(E^{'},E)}0$.

Then $(1)\Rightarrow(2)\Leftrightarrow(3)\Rightarrow(4)\Leftrightarrow(5)$ hold, and in addition, if $E$	has $STW^{*}C$ property, all are equivalent.
\end{theorem}
\begin{proof}
Since $x_\alpha^{'}\xrightarrow{st-uo}0\Leftrightarrow|x_\alpha^{'}|\xrightarrow{st-uo}0$, it is clear that $(2)\Leftrightarrow(3)\Rightarrow(4)\Leftrightarrow(5)$.

$(1)\Rightarrow(2)$ Assmue that $E$ has order continuous norm, by Theorem 2.1 of \cite{G:14}, for any statistical norm bounded lower finite net $(x^{'}_\alpha)$ and $x^{'}_\alpha\xrightarrow{st-uo}0$ in $E^{'}$, there exists a bounded subnet $(x^{'}_\beta)$ which asymptotic density of index set is $1$ satisfying $x^{'}_\beta\xrightarrow{uo}0$,then $x^{'}_\beta\xrightarrow{\sigma(E^{'},E)}0$, hence $x^{'}_\alpha\xrightarrow{st-\sigma(E^{'},E)}0$. 

$(4)\Rightarrow(1)$ For any norm bounded disjoint sequence $(x_n^{'})$ in $E^{'}$, obviously, it is statistical norm bounded and $uo$-null, hence it is $st$-$uo$-null, so $x^{'}_n\xrightarrow{st-\sigma(E^{'},E)}0$. Since $E$ has $STW^{*}C$ property, $(x_n^{'})$ is weak* convergent to zero. By Corollary 2.4.3 of \cite{MN:91}, $E$ has order continuous norm.
\end{proof}
\section{statistical unbounded convergence in Banach lattices}\label{}
A net $(x_\alpha)$ in Banach lattice $E$ is \emph{unbounded norm (resp. absolute weak, absolute weak*) convergent} to $x\in E$ if $(|x_\alpha-x|\wedge u)$ is norm (resp. weak, weak*) convergent for any $u\in E^+$ \cite{DOT:16,O:16,W:19}. We study the statistical unbounded convergence in Banach lattice.
\begin{definition}\label{}
A lower finite net $(x_\alpha)$ in Banach lattice $E$ is \emph{statistical unbounded norm convergent} to $x\in E$ provided that $\delta\{\alpha\in D:\big\Vert|x_\alpha-x|\wedge u\big\Vert\geq \epsilon\}=0$ for all $\epsilon>0$ and $u\in E_+$, we write $x_\alpha\xrightarrow{st-un}x$ and call $x$ the statistical unbounded norm limit of $(x_\alpha)$.
A lower finite net $(x_\alpha)$ in Banach lattice $E$ is \emph{statistical unbounded absolute weak convergent} to $x\in E$ provided that $\delta\{\alpha\in D:f(|x_\alpha-x|\wedge u)\geq \epsilon\}=0$ for all $\epsilon>0$, $f\in E^{'}_+$ and $u\in E_+$, we write $x_\alpha\xrightarrow{st-uaw}x$ and call $x$ the statistical unbounded absolute weak limit of $(x_\alpha)$.
A lower finite net $(x_\alpha^{'})$ in dual Banach lattice $E^{'}$ is \emph{statistical unbounded absolute weak* convergent} to $x^{'}\in E$ provided that $\delta\{\alpha\in D:(|x_\alpha^{'}-x^{'}|\wedge u^{'})(x)\geq \epsilon\}=0$ for all $\epsilon>0$, $x\in E_+$ and $u^{'}\in E_+$, we write $x_\alpha\xrightarrow{st-uaw^{*}}x$ and call $x^{'}$ the statistical unbounded absolute weak* limit of $(x_\alpha^{'})$.
\end{definition}

\begin{proposition}\label{}	
For a lower finite net $(x_\alpha)$ we have the following results (the $st$-$uaw$-convergence and $st$-$uaw^{*}$-convergence have analogic conclusions):

$(1)$ $st$-$un$-limits are unique;
	
$(2)$ suppose $x_\alpha\xrightarrow{st-un}x$ and $y_\alpha\xrightarrow{st-un}y$, then $ax_\alpha+by_\alpha\xrightarrow{st-un}ax+by$ for any $a,b\in R$;
	
$(3)$ $x_\alpha\xrightarrow{st-un}x$ iff $(x_\alpha-x)\xrightarrow{st-un}0$;
	
$(4)$ if $x_\alpha\xrightarrow{st-un}x$,then $|x_\alpha|\xrightarrow{st-un}|x|$;
\end{proposition}

Some subsequences of statistical unbounded convergent sequence is not necessarily convergent.
\begin{example}\label{}
In $L_p[0,1](1<p<+\infty)$, let 
\[x_n(t)=\begin{dcases}
0, t\notin[\frac{1}{2^n},\frac{1}{2^{n-1}}];\\
1, t\in[\frac{1}{2^n},\frac{1}{2^{n-1}}].
\end{dcases}  \]
and
\[y_n=\begin{dcases}
\sum_{1}^{n}ix_i, & n=k^2(k\in N);\\
2^nx_n, & ortherwise.
\end{dcases}  \]
Since $(x_n)$ is disjoint sequence in order continuous Banach lattice, by Corollary 3.6 of \cite{GTX:16} and Proposition 2.5 of \cite{DOT:16}, so it is unbounded norm (resp. absolute weak, absolute weak*) convergent to $0$, moreover, $(y_n)$ is statistical unbounded norm (resp. absolute weak, absolute weak*) convergent to $0$, but the subsequence $(y_{k^2})$ does not. We also can find that $(y_n)$ is not statistical norm (resp. absolute weak, absolute weak*) convergent and unbounded norm (resp. absolute weak, absolute weak*) convergent.
\end{example}

In Example 3.3, we know that $st$-$un$(resp. $uaw$, $uaw*$)-convergence is not necessarily $un$(resp. $uaw$, $uaw*$)-convergence. When does the $st$-$un$(resp. $uaw$, $uaw*$)-convergence imply $un$(resp. $uaw$, $uaw*$)-convergence? It is similar to Proposition 2.6 and 2.7, we have the following results.

\begin{lemma}\label{}
Let $(x_\alpha)$ be a lower finite net in Banach lattice $E$. Then $(x_\alpha)$ is $st$-$un(uaw, uaw^{*})$-convergent to $x\in E$ if and only if there is another lower finite net $(y_\alpha)$ such that $x_\alpha=y_\alpha$ for $a.a.\alpha$ and which is $un(uaw, uaw^{*})$-convergent to the same limit $x$.	
\end{lemma}
\begin{proposition}\label{}
For a monotonic lower finite net $(x_\alpha)$ in Banach lattice $E$, we have the following result: 

$(1)$ $un$-convergence and $st$-$un$-convergence agree.

$(2)$ $uaw$-convergence and $st$-$uaw$-convergence agree.

$(3)$ $uaw^{*}$-convergence and $st$-$uaw^{*}$-convergence agree.	
\end{proposition}
\begin{proof}
$(1)$ Suppose that $x_\alpha\uparrow$ and $x_\alpha\xrightarrow{un}x$. Since $x_\beta-x_\alpha\in E_+$ for all $\beta\geq \alpha$, because lattice operation of $un$-convergence is continuous, so $(x_\beta-x_\alpha)_{(\beta)}=(x_\beta-x_\alpha)^+_{(\beta)}\xrightarrow{un}(x-x_\alpha)^+$, hence $x-x_\alpha\geq0$ for all $\alpha$. If $x$ is not supremum, then assume that $y\in E$ such that $x_\alpha\leq y\leq x$ for all $\alpha\in D$, it follows that
$$x_\alpha=x_\alpha\wedge y\xrightarrow{un}x\wedge y=y$$
Hence, $y=x$, so $x=\sup_{\alpha\in D}x_\alpha$.  

According to proof of Proposition 2.7, if a lower finite net $(x_\alpha)$ in a Banach lattice $E$ is increasing and statistical unbounded norm converget to $x\in E$, then $x=\sup_{\alpha\in D}x_\alpha$.

Now, assume that $x_\alpha\xrightarrow{st-un}x$ and $x_\alpha\uparrow$, we have $x=\sup_{\alpha\in D}x_\alpha$. Hence there exists a $B\subset D$, all $\beta\in B$ and $\delta(B;D)=1$ such that $(x_\beta)$ is a subnet of $(x_\alpha)$ with $\sup_{\beta\in B}x_\beta=x$ and $x_\beta\xrightarrow{un}x$. For any $\epsilon>0$, $u\in E_+$, choose a $\beta_0$ such that $\big\Vert|x_{\beta_0}-x|\wedge u\big\Vert<\epsilon$. Then, for any $\alpha>\beta_0$, we have
$$0\leq|x-x_\alpha|\wedge u\leq|x-x_{\beta_0}|\wedge u$$
This shows
$$\big\Vert|x-x_\alpha|\wedge u\big\Vert\leq\big\Vert|x-x_{\beta_0}|\wedge u\big\Vert<\epsilon$$
Hence, $x_\alpha\xrightarrow{un}x$.

The proof of $(2)$ and $(3)$ is similar.
\end{proof}

In an order continuous Banach lattice, it is easy to check that $st$-$uo$-convergence implies $st$-$un$-convergence. So we can describe $KB$ space by statistical unbounded norm convergence. The following result is similar to Theorem 2.10.

\begin{theorem}\label{}
Let $E$ be an order continuous Banach lattice. The following are equivalent:
	
$(1)$ $E$ is $KB$ space;
	
$(2)$ every norm bounded $uo$-Cauchy sequence in $E$ is $un$-convergent;
	
$(3)$ every statistical norm bounded $st$-$uo$-Cauchy sequence in $E$ is $st$-$un$-convergent.
\end{theorem}

According to Example 3.3, $st$-$un$ (resp. $st$-$uaw$, $st$-$uaw^{*}$)-convergence is not necessarily $st$-$n$ (resp. $st$-$aw$, $st$-$aw^{*}$)-
convergence. It is natural to ask when the $st$-$un$ (resp. $st$-$uaw$, $st$-$uaw^{*}$)-convergence imply $st$-$n$ (resp. $st$-$aw$, $st$-$aw^{*}$)-
convergence? It is clear that the $st$-$un$ (resp. $st$-$uaw$, $st$-$uaw^{*}$)-convergence is the same as $st$-$n$(resp. $st$-$aw$, $st$-$aw^{*}$)-
convergence for statistical order bounded lower finite net.

By Theorem 2.3 of \cite{KMT:16}, the following holds.
\begin{proposition}\label{}
Let $E$ be a Banach lattice with strong order unit, then the lower finite net $(x_\alpha)$ in $E$ is $st$-norm convergent if and only if it is $st$-$un$-convergent.	
\end{proposition}

By Lemma 2 of \cite{KMT:16}, every disjoint lower finite net is $uaw$-null, hence it is $st$-$uaw$-null, moreover it is $st$-$uaw^{*}$-null. Using the result, we can describe order continuous and reflxive Banach lattice by statistical unbounded convergence.
\begin{theorem}\label{}
Let $E$ be an Banach lattice:
	
$(1)$ $E$ has order continuous norm;
	
$(2)$ for any statistical norm bounded lower finite net $(x^{'}_\alpha)$ in $E^{'}$,if $x^{'}_\alpha\xrightarrow{st-uaw^{*}}0$,then $x^{'}_\alpha\xrightarrow{st-\sigma(E^{'},E)}0$;

$(3)$ for any statistical norm bounded lower finite net $(x^{'}_\alpha)$ in $E^{'}$,if $x^{'}_\alpha\xrightarrow{st-uaw^{*}}0$,then $x^{'}_\alpha\xrightarrow{st-|\sigma|(E^{'},E)}0$;

$(4)$ for any statistical norm bounded sequence $(x^{'}_n)$ in $E^{'}$,if $x^{'}_n\xrightarrow{st-uaw^{*}}0$,then $x^{'}_n\xrightarrow{st-\sigma(E^{'},E)}0$;
	
$(5)$ for any statistical norm bounded sequence $(x^{'}_n)$ in $E^{'}$,if $x^{'}_n\xrightarrow{st-uaw^{*}}0$,then $x^{'}_n\xrightarrow{st-|\sigma|(E^{'},E)}0$.

Then $(1)\Rightarrow(2)\Leftrightarrow(3)\Rightarrow(4)\Leftrightarrow(5)$ hold, and in addition, if $E$	has $STW^{*}C$ property, all are equivalent.
\end{theorem}
\begin{proof}
Since $x_\alpha^{'}\xrightarrow{st-uaw^{*}}0\Leftrightarrow|x_\alpha^{'}|\xrightarrow{st-uaw^{*}}0$, it is clear that $(2)\Leftrightarrow(3)\Rightarrow(4)\Leftrightarrow(5)$.

$(1)\Rightarrow(2)$ Suppose that $E$ has order continuous norm, by Theorem 2.15 in \cite{W:19}, for any statistical norm bounded lower finite net $(x^{'}_\alpha)$ and $x^{'}_\alpha\xrightarrow{st-uaw^{*}}0$ in $E^{'}$, there exists a subnet $(x^{'}_\beta)$ which asymptotic density of index set is $1$ satisfying $x^{'}_\beta\xrightarrow{uaw^{*}}0$,then $x^{'}_\beta\xrightarrow{\sigma(E^{'},E)}0$, hence $x^{'}_\alpha\xrightarrow{st-\sigma(E^{'},E)}0$.

$(4)\Rightarrow(1)$ For any norm bounded disjoint sequence $(x_n^{'})$ in $E^{'}$, from the above, we know it is statistical norm bounded and $st$-$uaw^{*}$-null, so $x^{'}_n\xrightarrow{st-\sigma(E^{'},E)}0$. Since $E$ has $STW^{*}C$ property, $(x_n^{'})$ is weak* convergent to zero. By Corollary 2.4.3 of \cite{MN:91}, $E$ has order continuous norm.
\end{proof}
The following result is the dual version of Theorem 3.8.
\begin{theorem}\label{}
Let $E$ be an Banach lattice:
	
$(1)$ $E^{'}$ has order continuous norm;
	
$(2)$ for any statistical norm bounded lower finite net $(x^{'}_\alpha)$ in $E^{'}$,if $x_\alpha\xrightarrow{st-uaw}0$,then $x_\alpha\xrightarrow{st-\sigma(E,E^{'})}0$;

$(3)$ for any statistical norm bounded lower finite net $(x^{'}_\alpha)$ in $E^{'}$,if $x_\alpha\xrightarrow{st-uaw}0$,then $x_\alpha\xrightarrow{st-|\sigma|(E,E^{'})}0$;

$(4)$ for any statistical norm bounded sequence $(x_n)$ in $E$,if $x_n\xrightarrow{st-uaw}0$,then $x_n\xrightarrow{st-\sigma(E,E^{'})}0$;
	
$(5)$ for any statistical norm bounded sequence $(x_n)$ in $E$,if $x_n\xrightarrow{st-uaw}0$,then $x_n\xrightarrow{st-|\sigma|(E,E^{'})}0$.

Then $(1)\Rightarrow(2)\Leftrightarrow(3)\Rightarrow(4)\Leftrightarrow(5)$ hold, and in addition, if $E$	has $STWC$ property, all are equivalent.
\end{theorem}

\begin{corollary}\label{}
For Banach lattice $E$, then $st$-$uaw$-convergence and $st$-$w$-convergence are agree for sequence, if one of the following conditions is valid:

$(1)$ $E$ is $AM$-space;

$(2)$ $E$ is atomic, $E$ and $E^{'}$ is order continuous.	
\end{corollary}
\begin{proof}
Since $AM$-space and atomic Banach lattice with order continuous norm have lattice operation weakly sequence continuous property, so $st$-$w$-convergence implies $st$-$uaw$-convergence. For the reverse, the dual space of $AM$-space is $AL$-space which is order continuous, hence $st$-$uaw$-convergence implies $st$-$w$-convergence.
\end{proof}

\begin{theorem}\label{}
For Banach lattice $E$, the following are equivalent:

$(1)$ $E$ is reflexive;

$(2)$ every statistical norm bounded $st$-$uaw$-Cauchy sequence in $E$ is statistical weakly convergent;

$(3)$ every statistical norm bounded $st$-$un$-Cauchy sequence in $E$ is statistical weakly convergent;	

$(4)$ every statistical norm bounded $st$-$uo$-Cauchy sequence in $E$ is statistical weakly convergent.		
\end{theorem}
\begin{proof}
$(1)\Rightarrow(2)$ Since $E$ is reflexive, then $E$ and $E^{'}$ are $KB$ space, so we have every statistical norm bounded $st$-$uaw$-Cauchy sequence in $E$ is statistical weakly Cauchy. Since $E$ is $KB$-space, so it is statistical weakly convergent.

$(2)\Rightarrow(1)$ We show that $E$ contains no lattice copies of $c_0$ or $l_1$. Suppose that $E$ contains a sublattice isomorphic to $l_1$. The $(x_n)$ of $l_1$ which is in Example 2.5, it is $st$-$uaw$ convergent to zero and statistical norm bounded but it is not statistical weakly convergence.
Suppose that $E$ contains a sublattice isomorphic to $c_0$, the $(x_n)$ of $c_0$ which is in the proof of Theorem 2.10, it is $st$-$uaw$-Cauchy sequence, but it is not weakly convergence, so we have the conclusion. 

$(2)\Rightarrow(3)$ Since $un$-convergence implies $uaw$-convergence, so we have the conclusion.

$(3)\Rightarrow(1)$ $E$ contains no sublattice isomorphic to $c_0$ or $l_1$, the proof is similarily to $(2)\Rightarrow(1)$.

$(1)\Leftrightarrow(4)$ the proof is similar to $(1)\Leftrightarrow(2)$.
\end{proof}

\end{document}